\documentclass[11pt,reqno]{amsart}

\usepackage{graphics,graphicx,fontenc}
\usepackage{epsfig,psfrag,manfnt}
\usepackage{bbm,subfigure,rotating,bm,float,mathdots,wasysym}

\usepackage[all,cmtip]{xy}
\usepackage{amssymb,amsmath,amsfonts,amsthm,color}

\newcommand{\calD}{\mathcal{D}}

\newcommand{\calO}{\mathcal{O}}

\newcommand{\calS}{\mathcal{S}}

\newcommand{\mC}{\mathbb{C}}
\newcommand{\mD}{\mathbb{D}}

\newcommand{\mN}{\mathbb{N}}

\newcommand{\mR}{\mathbb{R}}

\newcommand{\mZ}{\mathbb{Z}}

\newcommand{\bba}{\mathbf{a}}
\newcommand{\bbb}{\mathbf{b}}
\newcommand{\bbc}{\mathbf{c}}
\newcommand{\bbd}{\mathbf{d}}
\newcommand{\bbe}{\mathbf{e}}
\newcommand{\bbf}{\mathbf{f}}
\newcommand{\bbg}{\mathbf{g}}
\newcommand{\bbh}{\mathbf{h}}

\newcommand{\bbn}{\mathbf{n}}

\newcommand{\bbp}{\mathbf{p}}
\newcommand{\bbq}{\mathbf{q}}

\newcommand{\bbu}{\mathbf{u}}
\newcommand{\bbv}{\mathbf{v}}

\newcommand{\bbx}{\mathbf{x}}
\newcommand{\bby}{\mathbf{y}}

\newcommand{\bbA}{\mathbf{A}}
\newcommand{\bbB}{\mathbf{B}}
\newcommand{\bbC}{\mathbf{C}}

\newcommand{\bbE}{\mathbf{E}}

\newcommand{\bbI}{\mathbf{I}}

\newcommand{\balpha}{\bm{\alpha}}
\newcommand{\bbeta}{\bm{\beta}}
\newcommand{\bgamma}{\bm{\gamma}}
\newcommand{\bdelta}{\bm{\delta}}

\newtheorem{theorem}{Theorem}[section]
\newtheorem{lemma}[theorem]{Lemma}
\newtheorem{corollary}[theorem]{Corollary}
\newtheorem{proposition}[theorem]{Proposition}

\theoremstyle{definition}

\newtheorem{remark}[theorem]{Remark}
\newtheorem{question}[theorem]{Question}

\newcommand{\nm}{\,\rule[-.6ex]{.13em}{2.3ex}\,}
\theoremstyle{definition}

\theoremstyle{definition}

\theoremstyle{definition}
\newtheorem{example}[theorem]{Example}

\begin{document}



\title[Algebraic properties of the periodic distribution ring]{A
  potpourri of algebraic properties \\of the ring of periodic
  distributions}

\author[A. Sasane]{Amol Sasane}
\address{Department of Mathematics \\London School of Economics\\
    Houghton Street\\ London WC2A 2AE\\ United Kingdom}
\email{sasane@lse.ac.uk}

\begin{abstract} 
  The set of periodic distributions, with usual addition and
  convolution, forms a ring, which is isomorphic, via taking a Fourier
  series expansion, to the ring $\calS'(\mZ^d)$ of sequences of at
  most polynomial growth with termwise operations.  In this article,
  we establish several algebraic properties of these rings.
\end{abstract}

\maketitle

\section{Introduction}

Purely algebraic properties for rings naturally considered in
Analysis, Algebraic Geometry or Operator Theory, have proven to be of
significant motivational importance behind theory-building in these
areas. For example, the Noetherian property for polynomial rings over
a Noetherian ring is the celebrated Hilbert Basis Theorem, which is a
cornerstone result in Algebraic Geometry.  As a second example,
Serre's 1955 question of whether the ring $k[x_1,\cdots, x_n]$ ($k$ a
field) is a projective-free ring spurred the development of algebraic
$K$-theory. As a third example, we mention the corona problem: given
data $a,b$ in the Hardy algebra $H^\infty(\mD)$ of bounded holomorphic
functions in the unit disk $\mD$ in $\mC$, Kakutani's 1941 question of
whether the pointwise corona condition $|a(z)|+|b(z)|>\delta$
($z\in \mD$) is sufficient for $H^\infty(\mD)$ to be equal to the
ideal $\langle a,b\rangle$ generated by $a,b$, led to huge advances in
Complex Analysis, Function-Theoretic Operator Theory, and Harmonic
Analysis through Carleson's 1962 solution to the problem.  Moreover,
specific algebraic properties possessed by rings arising in various
subdomains in Mathematics can lead to further advances in the theory. For example, Kazhdan's Property (T)
can be established for the special linear group over the ring
$\calO(X)$ holomorphic functions by investigating when the special
linear group over $\calO(X)$ can be generated by elementary matrices.

The theme of this article is to consider a naturally arising such ring
in Harmonic Analysis and Distribution Theory, namely the ring of
periodic distributions, and check which key algebraic properties are
possessed by this ring, and which ones aren't. Via a Fourier series
expansion, the ring $\calD'_{\mathbf{V}}(\mR^d)$ of periodic
distributions (with usual addition and convolution) is isomorphic to
the ring $\calS'(\mZ^d)$ of sequences of at most polynomial growth
with termwise operations, and we recall this below. We will use this
in all of our proofs.

\subsection{The ring $\calD'_{\mathbf{V}}(\mR^d)$ of periodic
  distributions.  \phantom{and the ring aaaa of aaa} \phantom{aaa} The
  ring $\calS'(\mZ^d)$ of Fourier coefficients of elements of
  $\calD_{\mathbf{V}}'(\mR^d)$}
\label{subsec_2}

\noindent For background on periodic distributions and its Fourier
series theory, we refer the reader to the books \cite[Chapter~16]{Dui}
and \cite[p.527-529]{Tre}.

Consider the space $\calS'(\mZ^d)$ of all complex valued maps on
$\mZ^d$ of at most polynomial growth, that is,
$$
\calS'(\mZ^d):=\bigg\{\mathbf{a}: \mZ^d\rightarrow \mC\;\Big|\;
\begin{array}{ll}\exists M>0 \; \exists k\in \mN \;\textrm{ such that} \\
\forall \bbn\in \mZ^d, \;
|\mathbf{a}(\bbn)|\leq M (1+\nm \bbn\nm)^k
\end{array}\bigg\},
$$
where $\nm \bbn\nm:=|n_1|+\cdots+|n_d|$ for all
$\bbn=(n_1,\cdots, n_d)\in \mZ^d$. Then $\calS'(\mZ^d)$ is a unital
commutative ring with pointwise operations, and the multiplicative
unit element given by the constant function $ \bbn\mapsto 1$, for all
$\bbn\in \mZ^d$.  The set $\calS'(\mZ^d)$ equipped with pointwise
operations, is a commutative, unital ring.  Moreover,
$(\calS'(\mZ^d),+,\cdot)$ is isomorphic as a ring, to the ring
$(\calD'_{\mathbf{V}}(\mR^d), +, \ast)$, where
$\calD'_{\mathbf{V}}(\mR^d)$ is the set of all periodic distributions
(see the definition below), with the usual pointwise addition of
distributions, and multiplication taken as convolution of
distributions.

For ${\mathbf{v}}\in {{\mathbb{R}}}^{d}$, the 
{\em translation operator} ${\mathbf{S_\bbv}}:\calD'(\mR^d)\rightarrow \calD'(\mR^d)$,  is defined by 
$$
\langle
{\mathbf{S_\bbv}}(T),\varphi\rangle = \langle
T,\varphi(\cdot+{\mathbf{v}})\rangle\textrm{ for all }\varphi \in
{\mathcal{D}}({\mathbb{R}}^d).
$$ 
A distribution $T\in {\mathcal{D}}'({\mathbb{R}}^d)$ is called {\em
  periodic with a period}
$\mathbf{v}\in {\mathbb{R}}^d\setminus \{\mathbf{0}\}$ if
$$
T= {\mathbf{S}_\bbv}(T).
$$
Let $ {\mathbf{V}}:=\{{\mathbf{v_1}}, \cdots, {\mathbf{v_d}}\} $ be a
linearly independent set of $d$ vectors in ${\mathbb{R}}^d$.  We
define ${\mathcal{D}}'_{{\mathbf{V}}}({\mathbb{R}}^d)$ to be the set
of all distributions $T$ that satisfy
 $$
 {\mathbf{S_{v_k}}}(T)=T, \quad
k=1,\cdots, d.  
$$ 
From \cite[\S34]{Don}, $T$ is a tempered distribution, and from the
above it follows by taking Fourier transforms that
$ (1-e^{2\pi i {\mathbf{v_k}} \cdot \mathbf{y}})\widehat{T}=0$, for
$ k=1,\cdots,d.  $ It can be seen that
$$
\widehat{T}
=
\sum_{\mathbf{v} \in V^{-1} {\mathbb{Z}}^d} 
\alpha_{\mathbf{v}}(T) \delta_{\mathbf{v}},
$$
for some scalars $\alpha_{\mathbf{v}}(T)\in {\mathbb{C}}$, and where
$V$ is the matrix with its rows equal to the transposes of the column
vectors ${\mathbf{v_1}}, \cdots, {\mathbf{v_d}}$:
$$
V:= \left[ \begin{array}{ccc} 
    \mathbf{v_1}^{\top} \\ \vdots \\ \mathbf{v_d}^{\top} 
    \end{array}\right].
$$
Also, in the above, $\delta_{\mathbf{v}}$ denotes the usual Dirac
measure with support in $\mathbf{v}$: 
$$
\langle \delta_{\mathbf{v}},
\varphi\rangle =\varphi (\mathbf{v}),\quad \varphi \in
{\mathcal{D}}({\mathbb{R}}^d).  
$$ 
Then the Fourier coefficients $\alpha_{ \bbv}(T)$ give rise to an
element in $\calS'(\mZ^d)$, and vice versa, every element in
$\calS'(\mZ^d)$ is the set of Fourier coefficients of some periodic
distribution. In this manner, the ring
$ (\calD'_{\mathbf{V}}(\mR^d),+,\ast) $ of periodic distributions on
$\mR^d$ is isomorphic (as a ring) to $ (\calS'(\mZ^d),+,\cdot).  $

The outline of this article is as follows: in the subsequent sections,
we will show that the ring $\calS'(\mZ^d)$ (and hence also the
isomorphic ring $\calD_{\mathbf{V}}(\mR^d)$) has the following
algebraic properties:
\begin{enumerate}
 \item $\calS'(\mZ^d)$ is not Noetherian. 
 \item $\calS'(\mZ^d)$ is a B\'ezout ring. 
 \item $\calS'(\mZ^d)$ is coherent. 
 \item $\calS'(\mZ^d)$ is a Hermite ring. 
 \item $\calS'(\mZ^d)$ is not projective-free. 
 \item $\calS'(\mZ^d)$ is a pre-B\'ezout ring. 
 \item For all $m\in \mN$, $SL_m(\calS'(\mZ^d))$ is generated by
   elementary matrices, that is,
   $SL_m(\calS'(\mZ^d))=E_m(\calS'(\mZ^d))$.
 \item A generalized ``corona-type pointwise condition'' on the matricial data
   $(A,b)$ with entries from $\calS'(\mZ^d)$ for the solvability of
   $Ax=b$ with $x$ also having entries from $\calS'(\mZ^d)$.
\end{enumerate}
In each section, we will first give the background of the algebraic
property, by recalling key definitions/characterizations, and then
prove the property, possibly with additional commentary.

\section{Noetherian property}

Recall that a commutative ring is called {\em Noetherian} if every
ascending chain of ideals is stationary, that is, given any chain of
ideals in the ring:
 $$
 I_1 \subset I_2\subset I_3\subset \cdots,
 $$
 there exists an $K\in \mN$ such that $I_K=I_{K+1}=\cdots$. 
  
\begin{proposition}
\label{prop_not_noetherian}
 $\calS'(\mZ^d)$ is not Noetherian.
\end{proposition}
\begin{proof} 
  For $k\in \mN$, set
  $ I_k=\{\bba \in \calS'(\mZ^d): \bba(\bbn)=0 \textrm{ for all }\nm
  \bbn\nm >k\}.  $
  Then $I_k$ is clearly an ideal in $\calS'(\mZ^d)$. Also, by
  considering the sequence
$$
\bbe_k:= \bigg(\mZ^d\owns \bbn  \mapsto \left\{ \begin{array}{ll} 1 &\textrm{if }\bbn=(k,0,\cdots, 0),\\
                                                  0
                                                                    &\textrm{otherwise},
                             \end{array}\right\} \bigg)\in \calS'(\mZ^d),
$$
for $k\in \mN$, we see that $\bbe_k\in I_k\setminus I_{k-1}$. So we
have the strict inclusions
$$
I_1\subsetneq I_2 \subsetneq I_3\subsetneq\cdots,
$$
showing the existence of an infinite ascending non-stationary chain of
ideals. Hence $\calS'(\mZ^d)$ is not Noetherian.
\end{proof}

\begin{remark}
We remark that in the same manner, one can also show that 
$$
\ell^\infty(\mZ^d):=\bigg\{\mathbf{a}: \mZ^d\rightarrow \mC\;\Big|\;
\begin{array}{ll}\exists M>0 \;\;\textrm{ such that} \\
\forall \bbn\in \mZ^d, \;
|\mathbf{a}(\bbn)|\leq M 
\end{array}\bigg\},
$$
the ring of all bounded sequences with pointwise operations, is not
Noetherian either.
 \end{remark} 

\section{B\'ezout ring} 

A commutative ring is called {\em B\'ezout} if every finitely
generated ideal is principal.

\begin{theorem}
 \label{prop_Bezout}
 Every finitely generated ideal in $\calS'(\mZ^d)$ is principal, that
 is, $\calS'(\mZ^d)$ is B\'ezout ring.
\end{theorem}

\noindent Before we give the proof of the above result, we collect some useful
observations first. For a complex sequence
$\bba=(\mZ^d\owns \bbn \mapsto \bba(\bbn))$, let
$$
|\bba|(\bbn):=|\bba(\bbn)|,\quad \bbn\in \mZ^d. 
$$ 
Then we can write $\bba=|\bba|\cdot
\bbu_{\bba}$, where 
$$
\bbu_{\bba}( \bbn)=\left\{\begin{array}{cl}
            \displaystyle \frac{\bba(\bbn)}{|\bba(\bbn)|} & \textrm{if } \bba(\bbn)\neq 0,\\
            1 & \textrm{if } \bba(\bbn)=0. \phantom{\displaystyle a^f}
          \end{array}\right. 
$$
Then $\bbu_\bba\in \calS'(\mZ^d)$. Also, $\bba\in \calS'(\mZ^d)$ if and only if
$|\bba|\in \calS'(\mZ^d)$. 
For a complex sequence $\bba=(\mZ^d\owns \bbn \mapsto \bba(\bbn))$, let 
$$
(\bba^\ast)(\bbn)=\bba(\bbn)^\ast,\quad \bbn\in \mZ^d,  
$$ 
where $\bba(\bbn)^\ast$ on the right hand side denotes the complex conjugate of the complex number $\bba(\bbn)$. 
Then $ \bba\in \calS'(\mZ^d)$ if and only if
$\bba^* \in \calS'(\mZ^d)$. Also,
$ \bbu_\bba \bbu_{\bba^*}={\mathbf{1}}$ (the constant sequence, taking
value $1$ everywhere on $\mZ^d$) and $|\bba|=\bba (\bbu_\bba)^*$.
 
\begin{proof}
  It is enough to show that an ideal $\langle \bba,\bbb\rangle$
  generated by $\bba,\bbb\in \calS'(\mZ^d)$ is principal.  We'll show
  that $\langle\bba,\bbb\rangle =\langle |\bba|+|\bbb|\rangle.  $
 
  Since $(\bbu_{\bba})^*, (\bbu_{\bbb})^*\in \calS'(\mZ^d)$, we have
  $|\bba|+|\bbb|=\bba (\bbu_{\bba})^*+\bbb (\bbu_{\bbb})^*\in \langle
  \bba,\bbb\rangle.  $
  Thus $\langle |\bba|+|\bbb|\rangle\subset \langle\bba,\bbb\rangle $.
 
\smallskip 
 
\noindent Define $\balpha$ by
$$
\alpha(\bbn) =\left\{\begin{array}{cl}
    \displaystyle \frac{\bba(\bbn)}{|\bba(\bbn)|+|\bbb(\bbn)|} & 
    \textrm{if } |\bba(\bbn)|+|\bbb(\bbn)|\neq 0,\\
    1 & \textrm{if } |\bba(\bbn)|+|\bbb(\bbn)|=0, 
    \phantom{\displaystyle a^F}
           \end{array}\right. 
$$
for all $\bbn\in \mZ^d$. Then $|\balpha(\bbn)|\leq 1$ for all $\bbn$, and so
$\balpha \in \calS'(\mZ^d)$. Moreover, $\bba=\balpha \cdot (|\bba|+|\bbb|)$,
and so $\bba\in \langle |\bba|+|\bbb|\rangle$.  Similarly, $\bbb\in
\langle |\bba|+|\bbb|\rangle$ too. Hence $\langle \bba,\bbb\rangle
\subset \langle |\bba|+|\bbb|\rangle$.

 Consequently, $\langle\bba,\bbb\rangle =\langle
|\bba|+|\bbb|\rangle$. This completes the proof.
\end{proof}

\section{Coherence}
\label{subsec_1}

A commutative unital ring $R$ is called {\em coherent} if every 
finitely generated ideal $I$ is finitely presentable, that is, 
there exists an  exact sequence 
$$
0\longrightarrow K \longrightarrow F\longrightarrow I \longrightarrow 0,
$$
where $F$ is a finitely generated free $R$-module and $K$ is a
finitely generated $R$-module.

We refer the reader to the monograph \cite{Gla} for
background on coherent rings and for the relevance of the property of
coherence in homological algebra. All Noetherian rings are coherent,
but not all coherent rings are Noetherian.  For example, the
polynomial ring $\mC[x_1,x_2,x_3,\cdots]$ is not Noetherian (because
the sequence of ideals $ \langle x_1 \rangle\subset \langle x_1, x_2
\rangle\subset \langle x_1, x_2, x_3 \rangle \subset \cdots $ is
ascending and not stationary), but $\mC[x_1, x_2, x_3,\cdots]$ is
coherent \cite[Corollary~2.3.4]{Gla}. Some equivalent
characterizations of coherent rings are listed below:
\begin{enumerate}
\item \cite{Cha}; \cite[Theorem~2.0A, p.404]{Fai}: Let $R$ be a unital
  commutative ring. Let $n\in \mN:=\{1,2,3,\cdots\}$ and
  $F=(f_1,\cdots, f_n)\in R^n$. A {\em relation} $G$ on $F$, written
  $G\in F^\perp$, is an $n$-tuple $G=(g_1,\cdots, g_n)\in R^n$ such
  that $g_1 f_1+\cdots + g_n f_n=0.$ The ring $R$ is coherent if and
  only if for each $n\in \mN$ and each $F\in R^n$, the $R$-module
  $F^\perp$ is finitely generated.
\item \cite[Definition, p.41, p.44]{Gla}: 
Let $R$ be a commutative unital ring. An $R$-module $M$ is called a
{\em coherent $R$-module} if it is finitely generated and every
finitely generated $R$-submodule $N$ of $M$ is finitely presented,
that is, there exists an exact sequence 
$$
F_1 \longrightarrow F_0 \longrightarrow N \longrightarrow 0
$$
with $F_1, F_2$ both finitely generated, free $R$-modules. Recall that
an $R$-module is a {\em free $R$-module} if it is isomorphic to a
direct sum of copies of $R$. A commutative unital ring $R$ is coherent
if and only if $R$ is a coherent $R$-module.
\end{enumerate}

Although it is known that B\'ezout {\em domains} are automatically
coherent, we can't use this fact and Theorem~\ref{prop_Bezout}, since
$\calS'(\mZ^d)$ is {\em not} a domain: there exist nontrivial zero
divisors in $\calS'(\mZ^d)$.  For $\bba\in \calS'(\mZ^d)$, let
$Z(\bba)$ denote the zero set of $\bba$, that is,
$$
Z(\bba):=\{\bbn\in \mZ^d:\bba(\bbn)=0\}.
$$
Let $\mathbf{0}\in \calS'(\mZ^d)$ denote the constant map $\mZ^d\owns \bbn\mapsto 0$.

\begin{theorem}
\label{main_theorem}
 $\calS'(\mZ^d)$ is a coherent ring.
\end{theorem}
\begin{proof} Let $I$ be a finitely generated ideal in $\calS'(\mZ^d)$. 
 Then $I$ is principal, and so there exists an 
$\bba\in \calS'(\mZ^d)$ such that $I=\langle
\bba\rangle$. Let $K=\langle {\mathbf{1}}_{Z(\bba)}\rangle$, where
${\mathbf{1}}_{Z(\bba)}$
is the indicator function of the zero set of $\bba$, that is, for all $\bbn\in
\mZ^d$,
$$
\big({\mathbf{1}}_{Z(\bba)}\big)(\bbn)
:=
\left\{\begin{array}{cl}
            0 & \textrm{if } \bba(\bbn)\neq 0,\\
            1 & \textrm{if } \bba(\bbn)=0. 
           \end{array}\right. 
$$
Then ${\mathbf{1}}_{Z(\bba)}\in \calS'(\mZ^d)$.  Moreover, let
$\varphi: \calS'(\mZ^d)\rightarrow I$ be the ring homomorphism given by
$\varphi(\bbb)=\bba \bbb$, for $\bbb\in \calS'(\mZ^d)$.  Finally let
$F:=\calS'(\mZ^d)=\langle \mathbf{1} \rangle$. Then we will check that the
sequence
$$
\begin{array}{ccccccccc}
 0 & \longrightarrow & K   & \longrightarrow  & F &  \stackrel{\varphi}{\longrightarrow} & I & \longrightarrow &0 
\\
   &                 & \rotatebox[origin=c]{90}{=} & & \rotatebox[origin=c]{90}{=} && \rotatebox[origin=c]{90}{=}& & \\
   
 && \langle {\mathbf{1}}_{Z(\bba)}\rangle 
&& \calS'(\mZ^d) && \langle \bba \rangle &&
\end{array}
$$
is exact. The exactness at $K$ and $I$ is clear. So we only need to
show that
$$
(\ker \varphi:=)\; \{\bbb\in \calS'(\mZ^d): \bba \bbb=\mathbf{0}\}=
\langle {\mathbf{1}}_{Z(\bba)}\rangle .
$$
Since $ {\mathbf{1}}_{Z(\bba)}\in \ker \varphi$, it is clear that
$\langle {\mathbf{1}}_{Z(\bba)}\rangle\subset \ker \varphi$. It
remains to show the reverse inclusion. Suppose that $\bbb\in \ker \varphi$. 
Then $\bba(\bbn) \bbb(\bbn)=0$ for all $\bbn\in \mZ^d$. Now if
$\bba(\bbn)\neq 0$, then $\bbb(\bbn)=0$. Hence
$$
\bbb= {\mathbf{1}}_{Z(\bba)} \cdot \bbb \in 
\langle {\mathbf{1}}_{Z(\bba)}\rangle.
$$
So $\ker \varphi \subset \langle {\mathbf{1}}_{Z(\bba)}\rangle$ as well.  
\end{proof}

\smallskip 

\noindent {\bf Remark on the coherence of $\ell^\infty(\mZ^d)$:}
 The above proof of Theorem~\ref{main_theorem} carries over,
mutatis mutandis, to the ring $\ell^\infty(\mZ)$.  Thus we obtain the
result:

\begin{theorem} 
\label{thm_ell_infty}
 $\ell^\infty(\mZ^d)$ is a coherent ring. 
\end{theorem}

\noindent This {\em also} follows from a classical result of
Neville~\cite{Nev}, which gives a topological characterization of
coherence for the ring $C(X;\mR)$ of all real-valued continuous functions
on $X$.
  
\begin{proposition}[Neville] 
\label{prop_neville_real}$\;$

\noindent 
$C(X;\mR)$ is coherent if and only if $X$ is basically disconnected. 
\end{proposition}
  
\noindent A topological space $X$ is called 
{\em basically disconnected} if for each $f\in C(X;\mR)$, the {\em cozero set of $f$}, 
 $
\textrm{coz}(f):=\{x\in X:f(x)\neq 0\},
$ 
has an open closure. 

We will need the complex-valued version of the above result, which can be obtained from the 
 following observation.
  
\begin{lemma}
\label{lemma_complex_vs_real}
$C(X;\mC)$ is coherent if and only if $C(X;\mR)$ is coherent.
\end{lemma}
  
\noindent Here $C(X;\mC)$ denotes the ring of all complex valued
continuous functions on $X$.  We will use  \cite[Corollary~2.2.2 and 2.2.3, p.43]{Gla}, 
 quoted below. 
  
\begin{proposition} 
\label{prop_gla_1}$\;$

\noindent 
If {\em (1)} $R$ is a commutative unital ring,

\noindent \phantom{If }{\em (2)} $M,N$ coherent $R$-modules, and

\noindent \phantom{If }{\em (3)} 
$\varphi:M\rightarrow N$ a homomorphism, 

\noindent then $\ker \varphi$ is a
coherent $R$-module.
\end{proposition}

\begin{proposition}
\label{prop_gla_2}$\;$

\noindent 
Every finite direct sum of coherent modules is a coherent module. 
\end{proposition}
  
\begin{proof}(of Lemma~\ref{lemma_complex_vs_real}): 
  
\noindent 
(``If'' part). Suppose that $C(X;\mR)$ is a coherent ring. Let $n\in \mN$. 

\noindent Let 
 $
\bbf_1=\bba_1+i\bbb_1,\; \cdots,\; \bbf_n=\bba+i\bbb_n\in C(X;\mC),
$ 
where each $\bba_j,\bbb_j\in C(X;\mR)$.  

\noindent Set $R:=C(X;\mR)$, 
$M:=C(X;\mR)^{(2n)\times  1}$,  and $N:=C(X;\mR)^{2\times (2n)}$. 

\noindent Suppose that $\varphi:M\rightarrow N$ is the
module homomorphism given by multiplication by the matrix
$$
[\Phi]:=\left[\begin{array}{cr|c|cr}
                 \bba_1 & -\bbb_1 & \cdots & \bba_n & -\bbb_n\\
                 \bbb_1 & \bba_1 & \cdots & \bbb_n & \bba_n 
                \end{array}\right].
$$
By Proposition~\ref{prop_gla_2}, $M,N$ are coherent $C(X;\mR)$-modules,
since $C(X;\mR)$ is a coherent ring.  Next, by
proposition~\ref{prop_gla_1}, $\ker \varphi$ is a coherent
$C(X;\mR)$-module, and in particular, it is finitely generated, say by
$$
\left[\begin{array}{c} \bbc_1^{(k)} \\\bbd_1^{(k)} \\ \vdots \\ \bbc_n^{(k)} \\ \bbd_n^{(k)}
\end{array}\right],\quad k=1,\cdots, m.
$$
Let $\bbg_1=\balpha_1+i\bbeta_1,\cdots, \bbg_n=\balpha_n+i\bbeta_n$ (where each $\balpha_j,\bbeta_j\in C(X;\mR)$) 
be such that 
$$
\bbf_1 \bbg_1+\cdots +\bbf_n \bbg_n=\mathbf{0}.
$$
Then 
$$
[\Phi]\left[ \begin{array}{c} \balpha_1 \\\bbeta_1\\ \vdots \\\balpha_n \\\bbeta_n\end{array}\right]=
\mathbf{0},
$$
and so there exist $\bgamma_1,\cdots, \bgamma_m$ such that 
$$
\left[ \begin{array}{c} \balpha_1 \\\bbeta_1\\ \vdots \\\balpha_n \\\bbeta_n\end{array}\right]
=
\bgamma_1 
\left[\begin{array}{c} \bbc_1^{(1)} \\ \bbd_1^{(1)} \\ \vdots \\ \bbc_n^{(1)} \\ \bbd_n^{(1)}\end{array}\right]
+ \cdots +
\bgamma_m 
\left[\begin{array}{c} \bbc_1^{(m)} \\ \bbd_1^{(m)} \\ \vdots \\ \bbc_n^{(m)} \\ \bbd_n^{(m)}\end{array}\right].
$$
But then 
$$
 \left[\begin{array}{c}
        \bbg_1\\ \vdots\\\bbg_n
       \end{array}\right]
       =
       \bgamma^{(1)} \left[ \begin{array}{c} \bbc_1^{(1)}+ i \bbd_1^{(1)}\\\vdots \\ 
                             \bbc_n^{(1)}+ i \bbd_n^{(1)}
                            \end{array}\right]+\cdots+ 
                          \bgamma^{(m)} \left[ \begin{array}{c} \bbc_1^{(m)}+ i \bbd_1^{(m)}\\\vdots \\
                              \bbc_n^{(m)}+ i \bbd_n^{(m)}
                            \end{array}\right]         .
$$
Hence we see that $(\bbf_1,\cdots,\bbf_n)^\perp$ is contained in the $C(X;\mC)$-module generated by 
$$
\left[ \begin{array}{c} \bbc_1^{(1)}+ i \bbd_1^{(1)}\\\vdots \\ 
                             \bbc_n^{(1)}+ i \bbd_n^{(1)}
                            \end{array}\right],\cdots, 
           \left[ \begin{array}{c} \bbc_1^{(m)}+ i \bbd_1^{(m)}\\\vdots \\ 
                             \bbc_n^{(m)}+ i \bbd_n^{(m)}
                            \end{array}\right]   .
$$
It is also clear that each of the above columns belongs to
$(\bbf_1,\cdots,\bbf_n)^\perp$. Hence $(\bbf_1,\cdots,\bbf_n)^\perp$
also contains the $C(X;\mC)$-module generated by the above columns.
Consequently, $C(X;\mC)$ is a coherent ring.
       
\smallskip 

\noindent (``Only if'' part). Now suppose that $C(X;\mC)$ is a
coherent ring. Let $n\in \mN$ and
$$
\bbA:=(\bba_1,\cdots, \bba_n)\in C(X;\mR)^{1\times n}.
$$
Suppose that 
$$
\left[\begin{array}{c} \bbc_1^{(k)}+i\bbd_1^{(k)}\\\vdots\\ \bbc_n^{(k)}+i\bbd_n^{(k)}\end{array}\right], 
\quad k=1,\cdots, m,
$$
generate the $C(X;\mC)$-module $\bbA^\perp$, where each $\bbc^{(k)}_j,
\bbd^{(k)}_j\in C(X;\mR)$. Consider a $\bbB=(\bbb_1,\cdots,\bbb_n)\in
C(X;\mR)^{1\times n}$ such that
$$
\bba_1\bbb_1+\cdots +\bba_n\bbb_n=\mathbf{0}.
$$
Then there exist $\bbp^{(k)}, \bbq^{(k)}\in C(X;\mR)$, $k=1,\cdots ,m$
such that
$$
\left[ \!\begin{array}{c} \bbb_1\\\vdots \\\bbb_n\end{array}\!\right]
=
(\bbp^{(1)}+i\bbq^{(1)})
\left[\!\!\begin{array}{c} \bbc_1^{(1)}+i\bbd_1^{(1)}\\\vdots\\ \bbc_n^{(1)}+i\bbd_n^{(1)}\end{array}\!\!\right]
+ \cdots +
 (\bbp^{(m)}+i\bbq^{(m)})
\!\left[\!\begin{array}{c} \bbc_1^{(m)}+i\bbd_1^{(m)}\\\vdots\\ \bbc_n^{(m)}+i\bbd_n^{(m)}\end{array}\!\right]\!.
$$
Equating real parts, we obtain in particular that 
$$
\left[\!\begin{array}{c} \bbb_1\\\vdots \\\bbb_n\end{array}\!\right]\!
=
\bbp^{(1)}
\left[\!\begin{array}{c} \bbc_1^{(1)} \\ \vdots \\ \bbc_n^{(1)}\end{array}\!\right]
-\bbq^{(1)}
\left[\!\begin{array}{c} \bbd_1^{(1)}\\\vdots\\ \bbd_n^{(1)}\end{array}\!\right]
+ \cdots +
\bbp^{(m)}
\left[\!\begin{array}{c} \bbc_1^{(m)} \\\vdots\\ \bbc_n^{(m)}\end{array}\!\right]
-\bbq^{(m)}
\left[\!\begin{array}{c} \bbd_1^{(m)}\\\vdots\\ \bbd_n^{(m)}\end{array}\!\right]\!.
$$
Thus the $C(X;\mR)$-module $\bbA^\perp$ is contained in the $C(X;\mR)$-module
generated by the $2m$ vectors
$$
\left[\!\begin{array}{c} \bbc_1^{(1)} \\ \vdots \\ \bbc_n^{(1)}\end{array}\!\right], 
\left[\!\begin{array}{c} \bbd_1^{(1)}\\\vdots\\ \bbd_n^{(1)}\end{array}\!\right], 
 \cdots,
\left[\!\begin{array}{c} \bbc_1^{(m)} \\\vdots\\ \bbc_n^{(m)}\end{array}\!\right],
\left[\!\begin{array}{c} \bbd_1^{(m)}\\\vdots\\ \bbd_n^{(m)}\end{array}\!\right]\!.
$$
On the other hand each of these vectors also lie in the $C(X;\mR)$-module $\bbA^\perp$, which can be seen 
immediately by equating the real and imaginary parts in 
$$
\bba_1(\bbc_1^{(k)}+i\bbd_1^{(k)})+\cdots+\bba_n (\bbc_n^{(k)}+i\bbd_n^{(k)})=\mathbf{0},\quad k=1,\cdots,m.
$$
Hence the $C(X;\mR)$-module $\bbA^\perp$ is finitely
generated. Consequently, $C(X;\mR)$ is coherent too.
\end{proof}

\noindent In light of Neville's result,
Proposition~\ref{prop_neville_real}, the above gives:

\begin{corollary}
\label{corollary_neville_complex}$\;$

\noindent 
$C(X;\mC)$ is coherent if and only if $X$ is basically disconnected.
\end{corollary}

\noindent If $X$ is a topological space, then let $C_b(X;\mC)$ denote
the algebra of bounded continuous complex valued functions on $X$,
endowed with pointwise operations and the supremum norm:
$$
\|\bbf\|_\infty:=\sup_{x\in X}|\bbf(x)|,\quad \bbf\in C_b(X;\mC).
$$
Then $C_b(X;\mC)$ is a $C^*$-algebra, and its maximal ideal space is
$\beta X$, the Stone-\v{C}ech compactification of $X$. 

Let $\mZ^d$ be endowed with the usual Euclidean topology inherited
from $\mR^d$. Then the $C^*$-algebra
$\ell^\infty(\mZ^d)=C_b(\mZ^d;\mC)$ is isomorphic to
$C(\beta \mZ^d;\mC)$.  But the Stone-\v{C}ech compactification
$\beta\mZ^d$ of the discrete space $\mZ^d$ is extremally disconnected
(that is, the closure of every open set in it is open), see for
example \cite[\S6.3, p.450]{PorWoo}, and in particular, also basically
disconnected. Using Corollary~\ref{corollary_neville_complex},
Theorem~\ref{thm_ell_infty} follows:
$\ell^\infty(\mZ^d)=C_b(\mZ^d;\mC)=C(\beta\mZ^d;\mC)$ is a coherent
ring. This completes the alternative proof of the coherence of
$\ell^\infty(\mZ^d)$.

\bigskip 

\noindent {\bf Remark on the coherence of $c(\mZ^d)$:} 
Let $c(\mZ^d)$ be the subring of $\ell^\infty(\mZ^d)$ consisting of
all convergent complex sequences, that is, 
$$
c(\mZ^d)=
\bigg\{\mathbf{a}\in \ell^\infty(\mZ^d)\;\Big|
\begin{array}{ll}\exists L\in \mC \textrm{ such that }
\forall \epsilon>0 \; \exists N\in \mN \textrm{ such that}\\ \forall \bbn\in \mZ^d \textrm{ such that }\nm \bbn\nm >N,\;
|\mathbf{a}(\bbn)-L|< \epsilon
\end{array}\bigg\}.
$$
The $C^*$-algebra $c(\mZ^d)$ is isomorphic to $C(\alpha \mZ^d;\mC)$,
where $\alpha \mZ^d$ denotes the Alexandroff one-point
compactification of $\mZ^d$ (where $\mZ^d$ has the usual Euclidean
topology on $\mZ^d$ inherited from $\mR^d$). So in light of
Corollary~\ref{corollary_neville_complex}, the question of coherence
of $c(\mZ^d)$ boils down to investigating whether or not
$\alpha \mZ^d$ is basically disconnected.
  
\begin{theorem}$\;$

\noindent {\em (1)} 
$\alpha \mZ^d$ is not basically disconnected. 

\noindent {\em(2)} $c(\mZ^d)$ is not a coherent ring.
\end{theorem}
\begin{proof} (1) 
Firstly, the closed sets $F$ of $\alpha \mZ^d$ are of the form
\begin{enumerate}
\item $F$ is a finite set of integer tuples, or
\item $F=S\cup \{\infty\}$, where $S$ is an arbitrary subset of the
  integer tuples.
\end{enumerate}
From here it follows that the function $\bbf:\alpha \mZ^d\rightarrow \mC$
given by
$$
\bbf(\bbn)=\left\{ \begin{array}{cl} \displaystyle 0 & \textrm{if } \nm\bbn\nm\textrm{ is even or }\bbn=\infty,
\phantom{a_{g_{\displaystyle a_g}} }\\
    \displaystyle \frac{1}{|\bbn|} & \textrm{if } \nm\bbn\nm\textrm{ is odd},
                \end{array}\right. 
$$
is continuous. Indeed, if $K$ is any closed subset of $\mC$ not
containing $0$, then $\bbf^{-1}(K)$ cannot contain $\infty$ and it can
only contain finitely many integer tuples, making it closed in $\alpha
\mZ^d$. On the other hand, if $K$ is a closed subset of $\mC$ containing
$0$, then $\bbf^{-1}(K)$ contains $\infty$, making it closed.  Hence the
inverse images of closed sets under $K$ stay closed. So $\bbf\in
C(\alpha \mZ^d;\mC)$. However, the cozero set of $\bbf$ is
$$
\textrm{coz}(\bbf)=\{\bbn\in \alpha \mZ^d: \bbf(\bbn)\neq 0\}
=\{\bbn\in \mZ^d:\nm \bbn\nm   \textrm{ is odd}\},
$$
whose closure is $\{\bbn\in \mZ^d:\nm\bbn\nm\textrm{ is odd}\}\cup\{\infty\}$, which is
clearly not open in $\alpha \mZ^d$. Hence $\alpha \mZ^d$ is not basically
connected.

\medskip 

\noindent (2) It follows from Corollary~\ref{corollary_neville_complex}
that $c(\mZ^d)$ is not coherent.
\end{proof}

We remark that $c(\mZ^d)$ is not Noetherian since it is not even coherent. 

\section{$\calS'(\mZ^d)$ is Hermite}

A notion related to coherence is that of a Hermite ring; see for example \cite[p.1026]{Tol}. 
The study of Hermite rings arose naturally in the development of algebraic $K$-theory associated with 
Serre's conjecture \cite{Lam}. 

In the language of modules, a ring $R$ is {\em Hermite} if 
every finitely generated stably free $R$-module is free. 

It is known that a commutative unital B\'ezout ring having Bass stable rank $1$ is 
Hermite \cite{Zab}. It was shown in \cite{RupSas} that the Bass stable rank of 
$\calS'(\mZ^d)$ is $1$. As $\calS'(\mZ^d)$ is a B\'ezout ring (Proposition~\ref{prop_Bezout}), we have 
the following:

\begin{theorem}
 $\calS'(\mZ^d)$ is a Hermite ring.
\end{theorem}

\section{$\calS'(\mZ^d)$ is not a projective free ring}

A related stricter notion than that of being Hermite, is the notion of
a projective free ring.

 A commutative unital ring $R$ is {\em projective free} if every
 finitely generated projective $R$-module is free.

 Clearly every projective free ring is Hermite, but the converse may
 not hold. In fact $\calS'(\mZ^d)$ is such an example: we will show
 below that $\calS'(\mZ^d)$ is {\em not} projective free.  We will do
 this using the following characterization of projective free rings;
 see \cite{BRS}.

\begin{proposition}
  Let $R$ be a commutative unital ring. Then $R$ is projective free if
  and only if for every $n\in \mN$ and every $P\in R^{n\times n}$ such
  that $P^2=P$, there exists an integer $r\geq 0$, an
  $S\in R^{n\times n}$, and an $S^{-1}\in R^{n\times n}$ such that
  $SS^{-1}=I_n$ and
 $$
 P=S^{-1}\left[\begin{array}{cc} I_r & 0\\ 0 & 0\end{array}\right]S.
 $$
 \textrm{\em (Here $I_r$ denotes the $r\times r$ identity matrix in $R^{r\times r}$.)}
\end{proposition}

\begin{theorem}
 $\calS'(\mZ^d)$ is not a projective free ring.
\end{theorem}
\begin{proof} Let $R=\calS'(\mZ^d)$ be projective free. Let $P=\mathbf{p}\in R^{1\times 1}$ be given by  
$$
\bbp(\bbn) =\left\{ \begin{array}{ll} 1 & \textrm{if }\nm\bbn\nm \textrm{ is even},\\
                                      0 & \textrm{if }\nm\bbn\nm \textrm{ is odd}.
                    \end{array}\right.
$$
Then $P^2=P$. Since $R$ is projective free, 
it follows that there are an integer $r\geq 0$, an 
 $S\in R^{1\times 1}$, and an $S^{-1}\in R^{1\times 1}$ such that 
 $$
 P=S^{-1}DS,
 $$
 where, since $r$ can only be $0$ or $1$, we have respectively that
 $D=\mathbf{0}$ or $\mathbf{1}$.  But then $P=\mathbf{0}$ or
 $P=\mathbf{1}$, and either case is not possible. This contradiction
 shows that $\calS'(\mZ^d)$ is not projective free.
\end{proof}

\section{$\calS'(\mZ^d)$ is a pre-B\'ezout ring}

Let $R$ be a commutative ring. We say that $d\in R$ {\em divides}
$a\in R$, written $d|a$, if there exists an $\alpha \in R$ such that
$d\alpha=a$.  For $a,b\in R$, an element $d\in R$ is called a {\em
  greatest common divisor} (gcd) {\em of $a,b\in R$} if
  \begin{enumerate}
   \item $d|a$,
   \item $d|b$, and 
   \item whenever $d'\in R$ is such that $d'|a$ and $d'|b$, then also $d'|d$. 
  \end{enumerate}
A commutative ring is a {\em GCD ring} if every $a,b\in R$ possess a gcd. 
It follows from Theorem~\ref{prop_Bezout} that $\calS'(\mZ^d)$ is a GCD ring. 

A commutative ring is {\em pre-B\'ezout} if for all $a,b\in R$ possessing a gcd $d$, 
there exist $x,y\in R$ such that $d=xa+yb$. 

\begin{theorem}
 $\calS'(\mZ^d)$ is a pre-B\'ezout ring.
\end{theorem}

The proof of this result follows the proof of an analogous result of
Mortini and Rupp in the context of rings of continuous functions
\cite{MorRup}.  For $\bba\in \calS'(\mZ^d)$, let
$\sqrt{|\bba|}\in \calS'(\mZ^d)$ be defined by
$\sqrt{|\bba|}(\bbn):=\sqrt{|\bba(\bbn)|}$, $\bbn\in \mZ^d$.

\begin{proof} Let $\bba,\bbb\in \calS'(\mZ^d)$, and let $\bbd$ be a
  gcd of $\bba,\bbb$.  We first claim that
  $Z(\bbd)=Z(\bba)\bigcap Z(\bbb)$. If $\bbn_0\in Z(\bbd)$, then
  $\bbd(\bbn_0)=0$.  Let $\balpha, \bbeta \in \calS'(\mZ^d)$ be such
  that $\bba=\balpha \bbd$, $\bbb=\bbeta \bbd$. Then
\begin{eqnarray*}
 &&\bba(\bbn_0)= \balpha (\bbn_0)\bbd(\bbn_0)=\balpha(\bbn_0)\cdot 0=0, \textrm{ and}\\
 &&\bbb(\bbn_0)= \bbeta (\bbn_0) \bbd(\bbn_0)=\bbeta (\bbn_0) \cdot 0=0.
\end{eqnarray*}
So $Z(\bbd)\subset Z(\bba)\bigcap Z(\bbb)$. Now let $\bbh:=\sqrt{|\bba|}+\sqrt{|\bbb|}$.  
 Then $\bbh |\bba$. Indeed, if $\bbn \in \mZ^d$ is such that 
 $\bba(\bbn)\neq 0$, then with 
 $$
 \bbA(\bbn):=\frac{\bba(\bbn)}{\sqrt{|\bba(\bbn)|}+\sqrt{|\bbb(\bbn)|}},
 $$
 we have 
 $$
 |\bbA(\bbn)|=\frac{|\bba(\bbn)|}{\sqrt{|\bba(\bbn)|}+\sqrt{|\bbb(\bbn)|}}
\leq \frac{|\bba(\bbn)|}{\sqrt{|\bba(\bbn)|}}=\sqrt{|\bba(\bbn)|}\leq |\bbh(\bbn)|.
$$
If $\bbn \in \mZ^d$ is such that $\bba(\bbn)\neq 0$, then we set
$\bbA(\bbn)=0$. Then it is clear that $\bbA\in \calS'(\mZ^d)$.  Also,
$\bba=\bbA\bbh$, and so $\bbh|\bba$. Similarly, $\bbh|\bbb$.  As
$\bbd$ is a gcd of $\bba,\bbb$, we must have $\bbh|\bbd$. So
$\bbd=\bbh \bdelta$ for some $\bdelta \in \calS'(\mZ^d)$. Hence
$Z(\bbh)\subset Z(\bbd)$. But $Z(\bbh)=Z(\bba)\bigcap Z(\bbb)$.  Thus
$Z(\bba)\bigcap Z(\bbb)\subset Z(\bbd)$. This completes the proof of
our claim that $Z(\bbd)=Z(\bba)\bigcap Z(\bbb)$.
 
 As $\bdelta:=\sqrt{|\balpha|}+\sqrt{|\bbeta|}$ divides $\balpha$ and
 $\bbeta$, it follows that $\bbd\bdelta$ divides $\bbd \balpha=\bba$
 and $\bbd \bbeta=\bbb$. Since $\bbd$ is a gcd of $\bba,\bbb$, there
 must exist a $\bbq\in \calS'(\mZ^d)$ such that
 $\bbd\bdelta\bbq=\bbd$, that is
 $\bbd(\bdelta \bbq-\mathbf{1})=\mathbf{0}$.  Now let
 $Y:=\mZ^d \setminus Z(\bbd)$. For $\bbn\in Y$, $\bbd(\bbn)\neq 0$,
 and so $\bbd(\bdelta \bbq-\mathbf{1})=\mathbf{0}$ gives
 $\bdelta(\bbn)\bbq(\bbn)=1$ for all $\bbn \in Y$. But as
 $\bbq\in \calS'(\mZ^d)$, there exist positive $M,k$ such that
 $|\bbq(\bbn)|\leq M(1+\nm \bbn \nm)^{k}$ for all $\bbn\in \mZ^d$.  In
 particular, for $\bbn \in Y$, we obtain from the above that, for all
 $\bbn\in Y$,
 \begin{equation}
 \label{eq_31_3_2017_1902}
 |\balpha(\bbn)|+|\bbeta(\bbn)|
 \geq 
 \frac{(\sqrt{|\balpha(\bbn)|}+\sqrt{|\beta(\bbn)|})^2}{2}
 = \frac{|\bdelta(\bbn)|^2}{2}
 \geq 
 \frac{1}{2M^2(1+\nm \bbn\nm)^{2k}}.
 \end{equation}
 For any nonzero complex number $z$, let $\textrm{Arg}(z)\in (-\pi,\pi]$ be 
 the unique number such that $z=|z|e^{i\textrm{Arg}(z)}$, and when $z=0$, 
 we set $\textrm{Arg}(0):=0$. 
 Now for $\bbn\in Y$, define 
 $$
 \bbx(\bbn)=\frac{e^{-i \textrm{Arg}(\balpha(\bbn))}}{|\balpha(\bbn)|+|\bbeta(\bbn)|} 
 \textrm{ and } 
 \bby(\bbn)=\frac{e^{-i \textrm{Arg}(\bbeta(\bbn))}}{|\balpha(\bbn)|+|\bbeta(\bbn)|}.
 $$
 If $\bbn\not\in Y$, we set $\bbx(\bbn)=0=\bby(\bbn)$. Then we have
 for all $\bbn \in Y$ that
 $1=\balpha(\bbn)\bbx(\bbn)+\bbeta(\bbn)\bby(\bbn)$, and by
 multiplying throughout by $\bbd(\bbn)$, we obtain for {\em all}
 $\bbn \in \mZ^d$ that
 $\bbd(\bbn)=\bba(\bbn)\bbx(\bbn)+\bbb(\bbn)\bby(\bbn)$.  But from the
 estimate \eqref{eq_31_3_2017_1902}, and the definition of
 $\bbx,\bby$, we see that $\bbx,\bby$ are elements of
 $\calS'(\mZ^d)$. Consequently $\bbd=\bba \bbx+\bbb\bby$ in
 $\calS'(\mZ^d)$, completing the proof.
\end{proof}

\section{$SL_m(R)=E_m(R)$ for $R=\calS(\mZ^d)$} 

Let $R$ be a commutative unital ring and $m\in \mN$. Then we introduce
the following terminology and notation:

\medskip 

\noindent (1) $I_m$ denotes the $m\times m$ identity matrix in $R^{m\times m}$, 
that is the square matrix 

\noindent \phantom{(1)} with all diagonal entries 
equal to $1\in R$ and off-diagonal entries equal to 

\noindent \phantom{(1)} $0\in R$. 

\medskip 

\noindent  (2)  $SL_m(R)$ denotes the group of all $m\times m$ 
matrices $M$ whose entries are

\noindent \phantom{(2)} elements of $R$ and determinant $\det M=1$. 

\medskip 

 \noindent (3)  An {\em elementary matrix} $E_{ij}(\alpha)$ over 
$R$ has the form $E_{ij}=I_n+\alpha \mathbf{e}_{ij}$, 
where 
\begin{enumerate}
 \item $i\neq j$, 
 \item $\alpha \in R$, and  
 \item $\mathbf{e}_{ij}$ is the $m\times m$ matrix whose entry in the
   $i$th row and $j$th column is $1$, and all the other entries of
   $\mathbf{e}_{ij}$ are zeros.
\end{enumerate}

\noindent (4) $E_m(R)$ is the subgroup of $SL_m(R)$ generated by the
elementary matrices.

\medskip 

 \noindent A classical question in commutative algebra is the following: 

\begin{question}
\label{question_GVP}
For all $m\in \mN$,  is $SL_m(R)=E_m(R)$?
\end{question}

\noindent The answer to this question depends on the ring $R$. For
example, if the ring $R=\mC$, then the answer is ``Yes'', and this is
an exercise in linear algebra; see for example \cite[Exercise~18.(c),
page~71]{Art}. On the other hand, if $R$ is the polynomial ring
$\mC[z_1, \cdots, z_d]$ in the indeterminates $z_1, \cdots, z_d$ with
complex coefficients, then if $d=1$, then the answer is ``Yes'' (this
follows from the Euclidean Division Algorithm in $\mC[z]$), but if
$d=2$, then the answer is ``No'', and \cite{Coh} contains the
following example:
$$
\left[ \begin{array}{cc} 1+z_1 z_2 & z_1^2 \\
        -z_2^2 & 1-z_1 z_2 
       \end{array}\right] \in SL_2(\mC[z_1,z_2]) \setminus E_2(\mC[z_1,z_2]).
$$
(For $d\geq 3$, the answer is ``Yes'', and this is the $K_1$-analogue
of Serre's Conjecture, which is the Suslin Stability Theorem
\cite{Sus}.)  The case of $R$ being a ring of real/complex valued
continuous functions was considered in \cite{Vas}.  For the ring
$R=\mathcal{O}(X)$ of holomorphic functions on Stein spaces in
$\mC^d$, Question~\ref{question_GVP} was posed as an explicit open
problem by Gromov in \cite{Gro}, and was solved in \cite{IvaKut}.  It
is known that $SL_m(\ell^\infty(\mN))=E_m(\ell^\infty(\mN))$; see
\cite{MO}.
 
We adapt the proof from \cite{MO} for answering
Question~\ref{question_GVP} for $R=\ell^\infty(\mN)$, to answer this
question for $R=\calS'(\mZ^d)$. We'll prove below
Theorem~\ref{thm_SL_n_for_per_dist}, saying that
$SL_m(\calS'(\mZ^d))=E_m(\calS(\mZ^d))$.  For a matrix
$M=[m_{ij}]\in \mC^{m\times m}$, we set
 $$
 \|M\|_\infty:=\displaystyle\max_{1\leq i, j\leq }|m_{ij}|.
 $$
 Then $\|M_1M_2\|_\infty \leq m \|M_1\|_\infty \|M_2\|_\infty$ for
 $M_1,M_2\in \mC^{m\times m}$.  Let $S_m$ denote the symmetry group
 for a set with $m$ elements.  For $p\in S_m$, let $\textrm{sign}(p)$
 denote the sign of $p$.
 
\begin{lemma}
There exist maps 
\begin{eqnarray*}
&& m\mapsto \nu(m):\mN\rightarrow \mN,\\
&& m\mapsto C(m): \mN\rightarrow (0,\infty),\\
&& m\mapsto k(m): \mN \rightarrow \mN,
\end{eqnarray*}  
such that for every $m\in \mN$ and every $A\in SL_m(\mC)$, there exist 
elementary matrices $E_1(A),\cdots, E_{\nu(m)}(A)$ such that 
$$
A=E_1(A)\cdots E_{\nu(m)}(A),
$$
and $\|E_n(A)\|_\infty \leq C(m)(1+\|A\|_\infty)^{k(m)}$ for all $n=1,\cdots, \nu(m)$. 
\end{lemma}
\begin{proof} First we note that if $A$ is a square matrix with
  determinant $\pm 1$, then $\|A\|_\infty$ cannot be too
  small. Indeed, as
 $$
 \pm 1=\det A=\sum_{p\in S_m} (\textrm{sign }\!p)\cdot A_{1p(1)}\cdots A_{mp(m)},
 $$
 we have $\|A\|_\infty \geq \frac{1}{\sqrt[m]{m!}}.$
 
 Now let $A\in SL_m(\mC)$. Consider first the case that $|a_{11}|=\|A\|_\infty$. So 
 with $a=a_{11}$, we have 
 $$
  A  =\left[\begin{array}{c|c} 
 a & \ast \\ \hline 
 \ast & \ast \end{array}\right].
 $$
 Now we premultiply the above by 
 $$
 E_a=\left[\begin{array}{c|c}
 \begin{array}{lc} a^{-1} & 0 \\
 0 & a \end{array} & 0 \\ \hline 
 \phantom{\displaystyle\sum}0 \phantom{\displaystyle\sum}& I \end{array}\right].
 $$
 As 
 $$
 \left[
 \begin{array}{lc} a^{-1} & 0 \\
 0 & a \end{array} \right]
 =
 \left[ \begin{array}{lc} 1& 0 \\a^{-1} & 1 \end{array} \right]
 \left[ \begin{array}{lc} 1 & 1-a \\ 0 & 1 \end{array} \right]
 \left[ \begin{array}{rc} 1 & 0 \\ -1 & 1  \end{array} \right]
 \left[ \begin{array}{lc} 1 & 1-a^{-1} \\ 0 & 1 \end{array} \right],
 $$
 we see that $E_a$ is a product of four elementary matrices. We have now 
 $$
 E_a A = \left[\begin{array}{c|c} 
 1 & \ast \\ \hline 
 \ast & \ast \end{array}\right].
 $$
 Using the entry $1$ as a pivot, we can use it to make all other
 entries in the first row and first column equal to $0$. In other
 words, there exist elementary matrices
 $E_1^{(\textrm{r})} ,\cdots, E_{m-1}^{(\textrm{r})} ,
 E_1^{(\textrm{c})} ,\cdots, E_{m-1}^{(\textrm{c})} $ such that
  \begin{equation}
  \label{eq_Fri_31_March_2017_2018}
  E_{m-1}^{(\textrm{r})} \cdots E_1^{(\textrm{r})} E_a  A  
  E_1^{(\textrm{c})} \cdots E_{m-1}^{(\textrm{c})}= \left[\begin{array}{c|c} 
 1 & 0 \\ \hline 
 0 & A_{m-1} \end{array}\right].
 \end{equation}
 So we have used $m-1+4+m-1=2(m+1)$ elementary matrices to obtain this 
 reduction for $A$. 
 Moreover, we have control on the size of the elementary matrices 
 we have used in terms of the size of $A$: indeed,  
 \begin{eqnarray*}
 \| \textrm{each factor of }E_a \|_\infty &\leq &1+\max\{|a^{-1}|, |a|\} 
\\
&\leq &
1+\max\{ \|A\|_\infty, \sqrt[m]{m!}\},
\\ 
\|E_i^{(\textrm{r})} \|_\infty, \;\|E_i^{(\textrm{c})} \|_\infty
&\leq& \|A\|_{\infty} m^4  (1+\max\{ \|A\|_\infty, \sqrt[m]{m!}\})^4,
\end{eqnarray*}
for all $i=1,\cdots, m-1$. All this we've done assuming
$|a_{11}|= \|A\|_\infty$. If this was not the case, then by working in
the same manner as above with the entry $(i_*,j_*)$ such that
$|a_{i_*,j_*}|=\|A\|_\infty$, we obtain
 $$
 E_{m-1}^{(\textrm{r})} \cdots E_1^{(\textrm{r})} E_a A  
  E_1^{(\textrm{c})} \cdots E_{m-1}^{(\textrm{c})}=
  \left[\begin{array}{c|c|c}
         P & \begin{array}{c} 0\\ \vdots \\ 0 \end{array}  & Q \\ \hline 
         \begin{array}{ccc} 0& \cdots & 0  \end{array} & 1& \begin{array}{ccc} 0& \cdots & 0  \end{array} \\ \hline 
         R & \begin{array}{c} 0\\ \vdots \\ 0 \end{array}  & S
        \end{array}\right]=:A', 
        $$
        where 
        $$
        A_{m-1}=\left[\begin{array}{cc} P & Q \\ R & S\end{array}\right].
        $$
        Clearly $\det A_{m-1}=\pm 1$, and so we can continue this
        process by using the largest entry of $A_{m-1}$ and using that
        as a pivot in the matrix $A'$, till we obtain that
        $$
        E_f A E_b =P,
        $$
        where $P$ is a permutation matrix, and $E_f$ is a product of 
        $$
           ( m-1+4)+(m-2+4)+\cdots+(1+4)
           $$ elementary matrices, 
        and $E_b$ is a product of $(m-1)+(m-2)+\cdots+1$ elementary matrices. 
        Also $\det P=(\det E_f )(\det A )(\det E_b)=1\cdot 1\cdot 1=1$. 
        But since each of the $m!/2$ even permutation matrices, which belong to 
        $SL_m(\mC)$ can be expressed as a finite product of elementary matrices 
        with entries that are bounded by constants that depend only on $m$, 
        we see that our claim is true. 
\end{proof}
 
 \begin{theorem}
 \label{thm_SL_n_for_per_dist}
 For all $m\in \mN$, $SL_m(\calS'(\mZ^d))=E_m(\calS(\mZ^d))$. 
 \end{theorem}
 \begin{proof}  Suppose $\bbA\in SL_m(\calS'(\mZ^d))$. 
 For every $\bbn\in \mZ^d$, 
 \begin{equation}
 \label{6_4_2017_1_07}
 \bbA(\bbn)=E^{[1]}(\bbn)\cdots E^{[\nu(m)]}(\bbn),
 \end{equation}
 where $E^{[1]}(\bbn),\cdots ,E^{[\nu(m)]}(\bbn)$ are elementary matrices over $\mC$, 
 with 
 \begin{equation}
 \label{equation_6_4_17_1:15}
 \|E^{[j]}(\bbn)\|_\infty \leq C(m)(1+\|\bbA(\bbn)\|_\infty)^{k(m)}.
 \end{equation}
 An elementary matrix $I_m+\alpha \bbe_{ij}$ is said to be of ``type''
 $(i,j)$.  We know that there are $m^2-m$ different ``types'' of
 elementary matrices.  (We'd like to see $\bbA$ expressed as a product
 $\bbE^{[1]}\cdots \bbE^{[N]}$ of elements
 $\bbE^{[1]},\cdots, \bbE^{[N]}$ from $E_m(\calS'(\mZ^d))$.  In light
 of \eqref{6_4_2017_1_07}, it seems tempting to define
 $\bbE^{[1]}(\bbn)=E^{[1]}(\bbn)$ etc, but we note that this is not
 guaranteed to give an element $\bbE^{[1]} $ in $E_m(\calS'(\mZ^d))$
 because $\bbE^{[1]}(\bbn_1)=E^{[1]}(\bbn_1)$ may not be of the same
 type as $\bbE^{[1]}(\bbn_2)=E^{[1]}(\bbn_2)$ for distinct
 $\bbn_1, \bbn_2$.  To remedy this, the idea now is as follows.  We
 think of the labels of the types of elementary matrices, say
 $a_1,\cdots,a_{m^2-m}$, as an alphabet, and consider the long word
$$
\underbrace{(a_1\cdots a_{m^2-m}) (a_1\cdots a_{m^2-m}) \cdots(a_1\cdots a_{m^2-m}) }_{\nu(m) \textrm{ groups}}.
$$
And we create a longer, partly redundant, factorization of
$\bbA(\bbn)$ than the one given in \eqref{6_4_2017_1_07} using this
long word as explained below. Then the {\em same} sequence of row
operations on each $\bbA(\bbn)$ will produce $I_m$. So we'll be able
to factorize $\bbA$ into elementary matrices over $\calS'(\mZ^d)$,
``uniformly'' instead of ``termwise''.  We now give the technical
details below.)
 
 We factor 
 $$
 \bbA(\bbn)= 
\underbrace{\Big(E_1^{[1]}(\bbn)\cdots E_{m^2-m}^{[1]}(\bbn)\Big)
 \cdots
\Big(E_1^{[\nu(m)]}(\bbn)\cdots E_{m^2-m}^{[\nu(m)]}(\bbn)\Big) 
  }_{\nu(m) \textrm{ groups}},
$$
where in each of the $\nu(m)$ groupings, all the matrices are
identity, except possibly for one: so if we look at the $i$th
grouping, if $E^{[i]}(\bbn)$ is of type $\alpha_k$, then
\begin{eqnarray*}
E^{[i]}_k(\bbn)&=&E^{[i]}(\bbn),\\
 E^{[i]}_{\ell}(\bbn)&=&I_m \textrm{ for all }\ell\neq k.
 \end{eqnarray*}
 (If it happens that $E^{[i]}(\bbn)$ is itself identity,  then  
 we put all of the $E^{[i]}_\ell(\bbn)= I_m$ for all $\ell=1,\cdots,m^2-m$.) 
Now define $\bbE_j^{[i]}\in E_m(\calS'(\mZ^d))$ by 
 $$
 \bbE_j^{[i]}(\bbn)=E_{j}^{[i]}(\bbn), \quad i=1,\cdots , \nu(m), \quad 
 j=1,\cdots, m^2-m, \quad \bbn\in \mZ^d.
 $$
 (The fact that we have entries in $\calS'(\mZ^d)$ follows from the
 estimate given in \eqref{equation_6_4_17_1:15}.)  Then
 $$
 \bbA=
\underbrace{\Big(\bbE_1^{[1]}\cdots \bbE_{m^2-m}^{[1]}\Big)
 \cdots
\Big(\bbE_1^{[\nu(m)]}\cdots \bbE_{m^2-m}^{[\nu(m)]}\Big) 
  }_{\nu(m) \textrm{ groups}}.
$$
This completes the proof.
 \end{proof}
 
\bigskip 

\noindent {\bf Remark on $SL_m(R)=E_m(R)$ for all $m\in \mN$ when $R=c(\mZ^d)$:} $\;$

\noindent 
Using a result given below in Lemma~\ref{lemma_9}, which follows from
\cite[Lemma~9]{Vas}, we will show Theorem~\ref{thm_c_also_SLn}.

\begin{lemma}[\cite{Vas}]
\label{lemma_9}
Let $R$ be a commutative topological unital ring such that the set of
invertible elements of $R$ is open in $R$. Let $m\in \mN$. If
$C\in SL_m(R)$ is sufficiently close to $I_m$, then $C$ belongs to
$E_m(R)$.
\end{lemma}

 \begin{theorem}
 \label{thm_c_also_SLn}
 For all $m\in \mN$, $SL_m(c(\mZ^d))=E_m(c(\mZ^d))$. 
 \end{theorem}
 \begin{proof}
  Let $m\in \mN$ and $\bbA \in SL_m(c(\mZ^d))$. Suppose that $L_{ij}$ is the 
  limit of the matrix entry $\bbA_{ij}\in c(\mZ^d)$, and $L$ be the complex $m\times m$ 
  matrix with the entry $L_{ij}$ in $i$th row and $j$th column. Since $\det:\mC^{m\times m}\rightarrow \mC$ 
  is continuous, we have 
  $$
  \det L= \det \Big(\lim \bbA(\bbn)\Big)= 
  \lim\det \bbA(\bbn)=\lim  1=1.
  $$ 
  Let $\epsilon>0$. Then there exists a $N\in \mN$ such that for all $\bbn\in \mZ^d$ such that $\nm \bbn\nm >N$, 
  we have $\|\bbA(\bbn)-L\|_\infty <\epsilon$. Let $ \bbB\in SL_m(c(\mZ^d))$ be defined by 
  $$
  \bbB(\bbn)=\left\{\begin{array}{ll}
                     \bbA(\bbn) & \textrm{if }\nm \bbn \nm \leq N, \phantom{\displaystyle \sum_g}\\
                     L & \textrm{if } \nm \bbn \nm >N.
                    \end{array}\right.
  $$
  Since $SL_m(\mC)=E_m(\mC)$, it is clear that $L$, as well as the finite number of matrices 
  $\bbA(\bbn)$ with $\nm \bbn \nm \leq N$, can all be written as a product of elementary matrices. Hence 
  it follows that $\bbB \in E_m(c(\mZ^d))$. But 
  $$
  \bbB=\bbA+\bbB-\bbA=\bbA(\bbI+\bbA^{-1} (\bbB-\bbA)),
  $$
  where $\bbI(\bbn):=I_m$ and $\bbA^{-1}(\bbn)=(\bbA(\bbn))^{-1}$ for
  all $\bbn\in \mZ^d$. To complete the proof, it suffices to show that
  $\bbC:=\bbI+\bbA^{-1} (\bbB-\bbA)\in E_m(c(\mZ^d))$.  First note
  that as $\bbA, \bbB\in SL_m(c(\mZ^d))$, we have $ 1=\det \bbA(\bbn)$
  and $1=\det \bbB(\bbn)$ for all $\bbn$. As $\bbA \bbC=\bbB$, it
  follows that also $\det \bbC(\bbn)=1$, and so
  $\bbC\in SL_m(c(\mZ^d))$.  To show $\bbC\in E_m(c(\mZ^d))$, we will
  use Lemma~\ref{lemma_9} above, with $R=c(\mZ^d)$. As
  $R=c(\mZ^d)=C(\mZ^d;\mC)$ is a Banach algebra, the set of invertible
  elements in $R$ is an open subset of $R$.  We have
   $$
  \bbC-\bbI= \bbA^{-1} (\bbA-\bbB), 
  $$
  and since $\bbB$ could have been made as close to $\bbA$ as we liked
  ($\|\bbB-\bbA\|_\infty <\epsilon$, and $\epsilon>0$ was arbitrary),
  it follows that $\bbC$ can be made as close as we like to
  $\bbI$. Hence $\bbC\in E_m(c(\mZ^d))$ by Lemma~\ref{lemma_9}.
 \end{proof}

\section{Solvability of $\bbA\bbx=\bbb$} 

We will show the following:

\begin{theorem}
\label{main_thm_13_feb_2018_11:40}
Let $\bbA\in (\calS'(\mZ^d))^{m\times n}$, $\bbb\in  (\calS'(\mZ^d))^{m\times 1}$.

\noindent 
Then the following two statements are equivalent:
\begin{enumerate}
 \item There exists an $\bbx \in (\calS'(\mZ^d))^{n\times 1}$ such that $\bbA \bbx =\bbb$.
 \item There exists a $\delta>0$ and $k>0$ such that 
 $$
 \forall \bbn\in \mZ^d,\;\forall y\in \mC^m, \;\|(\bbA (\bbn))^\ast y\|_2\geq \delta(1+\nm \bbn\nm)^{-k}|\langle 
 y,\bbb(\bbn)\rangle_2|.
 $$
\end{enumerate}
\end{theorem}

\noindent Here $\langle \cdot ,\cdot\rangle_2$ denotes the usual Euclidean inner product on $\mC^k$, 
and $\|\cdot\|_2$ is the corresponding induced norm. 

\begin{lemma}
\label{theorem_q_o}
Let 
\begin{itemize}
 \item[(1)] $A\in \mC^{m\times n}$ and $b\in \mC^{m}$,
 \item[(2)] there exist a $\delta>0$ such that 
 $\forall y\in \mC^{m},\;  \|A^{*}y\|_2\geq \delta |\langle y,b\rangle_2 |.$
\end{itemize}
Then there exists an $x\in \mC^{n}$ such that $Ax=b$ with $\|x\|_2\leq 1 / \delta$.
\end{lemma}
\begin{proof} If $y\in \ker A^{*}$, then (2) yields $\langle y,b\rangle_2=0$. 
Thus $b\in (\ker A^{*})^{\perp}=\textrm{ran }A$. 

If $y \in \ker AA^*$, then $
\|A^* y\|_2^{2}=\langle A^* y , A^* y \rangle= \langle AA^* y , y \rangle=\langle 0 , y \rangle=0.$ 
Thus $A^*y=0$, and so $y\in \ker A^*=(\textrm{ran }A)^{\perp}$. Since we had
shown above that $b\in \textrm{ran } A$, we have $ \langle b, y \rangle=0$. But
the choice of $y\in \ker AA^*$ was arbitrary, and so $b\in (\ker AA^*)^\perp= \textrm{ran}(AA^*)^*=\textrm{ran}(AA^*)$. 
Hence there exists a $y_0\in \mC^{m}$ such that $AA^* y_0=b$. Taking $x:=A^*y_0 \in \mC^{n}$, we
have $Ax=b$.

If $b=0$, then we can take $x=0$, and the estimate on $\|x\|_2$ is
obvious. So we assume that $b\neq 0$ and so $A^* y_0 \neq 0$. 
We have
$$
\|A^*y_0\|^2 
=
\langle A^* y_0 , A^* y_0 \rangle =\langle y_0 , AA^*
y_0 \rangle = \langle y_0 , b \rangle=|\langle y_0 , b
\rangle|
\leq \frac{1}{\delta} \|A^* y_0\|_2.
$$
Since $A^*y_0\neq 0$, we obtain $\|x\|_2=\|A^*y_0\|_2 \leq 1 /\delta$.
\end{proof}

\begin{proof}(Of Theorem~\ref{theorem_q_o}:) 

\noindent (1)$\;\Rightarrow\;$(2): As $\bbx\in (\calS'(\mZ^d))^{n\times 1})$, there exist $M,k>0$ 
such that for all $\bbn\in \mZ^d$, $\|\bbx(n)\|_2 \leq M(1+\nm \bbn\nm)^k$. Thus for all 
$y\in \mC^m$ and all $\bbn\in \mZ^d$, 
\begin{eqnarray*}
 |\langle y,\bbb(\bbn)\rangle_2|&=&|\langle y,\bbA(\bbn)\bbx(\bbn)\rangle_2|
 =|\langle (\bbA(\bbn))^*y,\bbx(\bbn)\rangle_2|\\
 &\leq & \| (\bbA(\bbn))^*y\|_2 \|\bbx(\bbn)\|_2 \quad\quad\textrm{(Cauchy-Schwarz)}\\
 &\leq &\| (\bbA(\bbn))^*y\|_2  M(1+\nm \bbn\nm)^k .
\end{eqnarray*}
Setting $\delta:=1/M>0$ and rearranging gives (2). 

\bigskip 

\noindent (2)$\;\Rightarrow\;$(1): Fix $\bbn\in \mZ^d$. Then (2) gives 
$$
\forall y\in \mC^m, \;\|(\bbA (\bbn))^\ast y\|_2\geq \delta(1+\nm \bbn\nm)^{-k}|\langle 
 y,\bbb(\bbn)\rangle_2|.
 $$
 Lemma~\ref{theorem_q_o} immediately gives an $x\in \mC^n$ such that 
  $
 \bbA(\bbn) x=\bbb(\bbn),
 $ 
 with 
 \begin{equation}
  \label{eq_estimate_12_feb_2018_1608}
  \|x\|_2\leq \frac{1}{\delta(1+\nm \bbn\nm)^{-k}}.
 \end{equation}
 Now set $\bbx(\bbn):=x$. By changing $\bbn$ at the outset, we obtain in this manner a map 
 $\bbx:\mZ^d\rightarrow \mC^n$. Setting $M=1/\delta>0$, we have that $\bbx\in (\calS'(\mZ^d))^{n\times 1}$
 since we obtain from \eqref{eq_estimate_12_feb_2018_1608} that 
 $$
 \forall \bbn\in \mZ^d, \;\|\bbx(\bbn)\|_2\leq M(1+\nm \bbn\nm)^k.
 $$
 Moreover, $\bbA\bbx=\bbb$. This completes the proof.
\end{proof}

 \noindent  For $\ell^\infty(\mZ^d)$, one has the following analogous result, and he same proof goes through, mutatis mutandis:

\begin{theorem}
Let $\bbA\in (\ell^\infty(\mZ^d))^{m\times n}$, $\bbb\in  (\ell^\infty(\mZ^d))^{m\times 1}$.

\noindent 
Then the following two statements are equivalent:
\begin{enumerate}
 \item There exists an $\bbx \in (\ell^\infty(\mZ^d))^{n\times 1}$ such that $\bbA \bbx =\bbb$.
 \item There exists a $\delta>0$ and $k>0$ such that 
 $$
 \forall \bbn\in \mZ^d,\;\forall y\in \mC^m, \;\|(\bbA (\bbn))^\ast y\|_2\geq \delta |\langle 
 y,\bbb(\bbn)\rangle_2|.
 $$
\end{enumerate}
\end{theorem}

\noindent We have 
$$
\ell^\infty(\mZ^d)=C_b(\mZ^d;\mC)=C(\beta\mZ^d;\mC)
$$ 
is a Banach algebra. 
Moreover, the natural point evaluation complex homomorphisms 
$$
\ell^\infty(\mZ^d)\owns \bba \mapsto \bba(\bbn)\in \mC, 
$$
constitute a dense set in its maximal ideal space $\beta \mZ^d$. Based on this, 
one may naturally pose the following question:

\begin{question}$\;$

\noindent 
Let $R$ be a commutative, unital, complex, semisimple Banach algebra. 

\noindent 
Suppose that $D$ be a dense set in the maximal ideal space of $R$ with the usual Gelfand topology, 
and let $\widehat{\;\cdot\;}$ denote the Gelfand transform. 

\noindent Let $\bbA\in R^{m\times n}$, $\bbb\in R^{m\times 1}$. 

\noindent {\bf Are the following two statements are equivalent?} 
\begin{enumerate}
 \item There exists an $\bbx \in R^{n\times 1}$ such that $\bbA \bbx =\bbb$.
 \item There exists a $\delta>0$ such that  
 $$
 \forall \varphi\in D,\;\forall y\in \mC^m, \;\|(\widehat{\bbA} (\varphi))^\ast y\|_2\geq \delta |\langle 
 y,\widehat{\bbb}(\varphi)\rangle_2|.
 $$
\end{enumerate}
(Here $\widehat{\bbA}$, $\widehat{\bbb}$ denote the matrices comprising the entry-wise Gelfand transforms of 
$\bbA$, $\bbb$ respectively.)
\end{question}

\noindent It can be seen easily that (1)$\;\Rightarrow\;$(2) is true. However, we now show that (2)$\;\Rightarrow\;$(1) 
may not hold, by considering the case of $c(\mZ^d)= C(\alpha \mZ^d;\mC)$. 

\begin{example}
Let $d=1$, so that $\mZ^d=\mZ$, and 
$$
\bbA(n)=\left[ \begin{array}{cc} 1 & 1\\ \bba(n) & \bbb(n) \end{array}\right]\in \mR^{2\times 2}, 
\quad 
\bbb(n)=\left[ \begin{array}{cc} 1 \\ 0 \end{array}\right]\in \mR^{2\times 1}, \quad n\in \mZ,
$$
where the (real) sequences $\bba,\bbb\in c(\mZ)$ will be suitably constructed later. Taking the dense set 
$D=\mZ$ in the maximal ideal space $\alpha \mZ$ of $c(\mZ)$, the condition (2) above becomes:
$$
\forall n\in \mZ,\;\forall y=\left[ \begin{array}{cc} \alpha \\ \beta \end{array}\right]\in \mC^{2\times 1},
\; |\alpha+\beta \bba(n)|^2+ |\alpha+\beta \bbb(n)|^2 \geq \delta^2 |\alpha|^2.
$$
If $\alpha=0$, then this condition is trivially satisfied. 

\noindent If $\alpha\neq 0$, then dividing throughtout by $|\alpha|^2$, and setting 
$\beta/\alpha=re^{i\theta}$, where $r>0$ and $\theta\in \mR$, we obtain 
$$
\forall n\in \mZ,\;\forall r>0,\;\forall \theta \in \mR,
\; |1+re^{i\theta} \bba(n)|^2+ |1+re^{i\theta} \bbb(n)|^2 \geq \delta^2 ,
$$
that is, 
$$
\forall n\in \mZ,\;\forall r>0,\;\forall \theta \in \mR,
\; (\bba(n)^2+\bbb(n)^2)r^2+2(\bba(n)+\bbb(n))(\cos \theta)r+2-\delta^2  \geq 0.
$$
This will be satisfied for all $r,\theta,n$ if, viewed as a (quadratic) polynomial in $r$ (with $n,\theta$ fixed arbitrarily), 
it has no real roots or has coincident real roots, that is, if 
$$
\Delta:=4\big((\bba(n)+\bbb(n))^2(\cos \theta)^2- (\bba(n)^2+\bbb(n)^2)(2-\delta^2)\big)\leq 0.
$$
First of all, to ensure that we have a quadratic polynomial, we demand that 
\begin{equation}
\label{condition_A1_yeah}
\boxed{\forall n\in \mZ,\;\bba(n)^2+\bbb(n)^2\neq 0.}
\end{equation}
Set  $\delta=1$. Then 
\begin{eqnarray*}
\Delta/4&=&(\bba(n)+\bbb(n))^2(\cos \theta)^2- (\bba(n)^2+\bbb(n)^2)\\
&=& (\bba(n)+\bbb(n))^2- (\bba(n)^2+\bbb(n)^2)+\big((\cos\theta)^2-1\big) (\bba(n)+\bbb(n))^2
\\
&=& 2\bba(n)\bbb(n)+\big(\underbrace{(\cos\theta)^2-1}_{\leq 0}\big) (\bba(n)+\bbb(n))^2\leq 2\bba(n)\bbb(n).
\end{eqnarray*}
So we can ensure that $\Delta\leq 0$ by demanding that 
\begin{equation}
\label{condition_A2_yeah}
\boxed{\forall n\in \mZ,\;\bba(n)\cdot \bbb(n)\leq 0.}
\end{equation}
With $\bba,\bbb$ satisfying \eqref{condition_A1_yeah} and \eqref{condition_A2_yeah}, we have that condition (2) 
holds with $\delta=1$. 

We will now stipulate additional conditions on $\bba,\bbb$ so that $\bbA \bbx=\bbb$ does {\em not} 
possess a solution $\bbx\in (c(\mZ))^{2\times 1}$. To this end, we demand that 
$\det \bbA(n)\neq 0$ for all $n$, that is,
\begin{equation}
\label{condition_A3_yeah}
\boxed{\forall n\in \mZ,\;\bbb(n)-\bba(n)\neq 0.}
\end{equation}
Then the unique solution $\bbx(n)$ to $\bbA(n)\bbx(n)=\bbb(n)$ is given by 
$$
\bbx(n)=\left[ \begin{array}{cc} 1 & 1\\ \bba(n) & \bbb(n) \end{array}\right]^{-1}
\left[ \begin{array}{cc} 1 \\ 0 \end{array}\right]
=
\left[ \begin{array}{cc} \dfrac{\bbb(n)}{\bbb(n)-\bba(n)} \\ \dfrac{\bba(n)}{\bba(n)-\bbb(n)} \end{array}\right].
$$
We want to ensure that $\bbx:=(n\mapsto \bbx(n))$ does not belong to $(c(\mZ))^{2\times 1}$. 
This will be guaranteed if one of its entries is not a convergent sequence. So we demand, say, that the sequence 
\begin{equation}
\label{condition_A4_yeah}
\boxed{
\bigg( \frac{\bba(n)}{\bba(n)-\bbb(n)}\bigg)_{n\in \mN} \textrm{ does not converge}.}
\end{equation}
It remains to construct sequences $\bba, \bbb$ in $c(\mZ)$ possessing the properties 
\eqref{condition_A1_yeah}, \eqref{condition_A2_yeah}, 
\eqref{condition_A3_yeah}, and \eqref{condition_A4_yeah}. 
We may take, for example, 
$$
\bba(n)=\frac{1}{1+n^2}\quad  \textrm{and}\quad
\bbb(n)=-\frac{\{n\sqrt{2}\}}{1+n^2},\quad n\in \mZ,
$$
where $\{x\}:=x-\lfloor x\rfloor$ denotes the fractional part of a real number $x$.  
Then $\bba,\bbb\in c(\mZ)$ because 
$$
\lim_{|n|\rightarrow \infty}\bba(n)=0=\lim_{|n|\rightarrow \infty}\bbb(n).
$$
Condition \eqref{condition_A1_yeah} is satisfied since 
 $
\forall n\in \mZ,\; \bba(n)^2+\bbb(n)^2\geq \bba(n)^2> 0.
$

\noindent 
\eqref{condition_A2_yeah} is fulfilled as
$$
\forall n\in \mZ,\; \bba(n)\cdot \bbb(n)= -\frac{\{n\sqrt{2}\}}{(1+n^2)^2} \leq  0.
$$
Condition \eqref{condition_A3_yeah} holds because 
 $\displaystyle 
\forall n\in \mZ,\; \bba(n)- \bbb(n)= \frac{1+\{n\sqrt{2}\}}{1+n^2} >  0.
$

\noindent 
Finally, we check that \eqref{condition_A4_yeah} is satisfied too. We have 
$$
\frac{\bba(n)}{\bba(n)-\bbb(n)}=\frac{1}{1-\bbb(n)/\bba(n)}= \frac{1}{1+\{n\sqrt{2}\}} .
$$
By Kronecker's Equidistribution Theorem (see for e.g. \cite[p.106-107]{SteSha}),  
the set $\{\{n\sqrt{2}\}: n\in \mN\}$ is dense in $[0,1)$, and so there are subsequences $(\{n_k\sqrt{2}\})_{k\in \mN}$ 
and $(\{\widetilde{n}_k\sqrt{2}\})_{k\in \mN}$ that converge to $0$, respectively $1/2$, and so 
$$
\lim_{k\rightarrow \infty} \frac{\bba(n_k)}{\bba(n_k)-\bbb(n_k)}
=
\frac{1}{1+0}=1\neq \frac{2}{3}=\frac{1}{1+1/2}=\lim_{k\rightarrow \infty} 
\frac{\bba(\widetilde{n}_k)}{\bba(\widetilde{n}_k)-\bbb(\widetilde{n}_k)},
$$
contradicting the 
convergence of $
\bigg( \dfrac{\bba(n)}{\bba(n)-\bbb(n)}\bigg)_{n\in \mN}$. 
 \hfill$\Diamond$
\end{example}

\medskip 

\noindent {\bf Acknowledgements:} The author thanks the anonymous referees for their comments. 
In particular, the first referee for the careful review, and suggestions which improved 
the presentation of the article.

\end{document}